\documentclass[10pt,a4paper,reqno]{amsart}
\usepackage{amsthm,amssymb,amsmath,graphicx,enumerate}
\usepackage{mathrsfs,setspace,pstricks,multicol,latexsym}
\usepackage{epsfig, subcaption, graphics}
\usepackage{hyperref}
\usepackage{multirow}
\usepackage{dsfont}
\usepackage{float}
\usepackage{caption}

\numberwithin{equation}{section}
\newtheorem{theo}{{\bf{Theorem}}}[section]
\newtheorem{prop}[theo]{{\bf Proposition}}
\newtheorem{lem}[theo]{{\bf Lemma}}

\newtheorem{ex}[theo]{{\bf Example}}

\newtheorem{rem}[theo]{{\bf Remark}}

\newtheorem{remark}[theo]{Remark}
\newtheorem{defi}[theo]{{\bf Definition}}

\newtheorem{notation}[theo]{{\bf Notation}}

\renewcommand{\proof}{\noindent{\bf Proof.\ }}

\catcode`\@=11
\catcode`\@=12

\begin{document}
\title{On Merging of Stochastic Flow of Semi-Markov Dynamics}\thanks{This research was supported in part by NBHM 02011/1/2019/NBHM(RP)R\& D-II/585, DST/INTDAAD/P-12/2020, and DST FIST (SR/FST/MSI-105). The second author would like to acknowledge the support from CSIR SRF. The support and the resources provided by ‘PARAM Brahma Facility’ under the National Supercomputing Mission, Government of India at the Indian Institute of Science Education and Research (IISER) Pune are gratefully acknowledged.}

\author{Anindya Goswami*}
\address{IISER Pune, India}
\email{anindya@iiserpune.ac.in}
\thanks{*Corresponding author}

\author{Ravishankar Kapildev Yadav}
\address{IISER Pune, India}
\email{ravishankar.kapildevyadav@students.iiserpune.ac.in}

\maketitle

\begin{abstract}
\noindent Given a semi-Markov law, using an additional parameter, we consider a family of stochastic flows corresponding to that law. Then we suitably select a particular flow, for which we obtain expressions of the meeting and merging probabilities of a pair of semi-Markov processes, solving the same equation but having two different initial conditions. A set of sufficient conditions are also obtained under which any two solutions of the flow eventually merge with probability one.
\end{abstract}

{\bf AMSC} 60G55, 60H20, 60J90, 60K15

\section{Introduction}
\noindent We investigate a modern methodology for the analysis of semi-Markov processes (SMP), whose general theory dates back to the ’50s and ’60s \cite{levyp,smith1955regenerative,takacs1954some}. In a recent work \cite{elliott2020semimartingale} a semimartingale dynamics of semi-Markov chain appears in contrast to the traditional description of a semi-Markov chain in terms of a renewal process. This presentation is however different from that in \cite{Ghosh2009RiskMO,GhoshSaha2011,generalsemi} where another semimartingale representation appears using an integration with respect to a Poisson random measure(PRM). Using the representation of latter type, we have obtained explicit formulae for probabilities of various events related to the stochastic flow of  semi-Markov dynamics.

A comprehensive study of the flow may reveal various aspects of mixing, meeting and coalescence of the dynamics. For example the study of merging of one dimensional Brownian flow goes back to Arratia \cite{arratia1979coalescing}, and Harris \cite{HARRIS1984187} whereas, Melbourne and Terhesiu \cite{AIHP968} have studied mixing for a class of non Markov flows. However, as per our knowledge, questions regarding meeting and merging have not been addressed in the literature for stochastic flow of SMPs. We find the questions related to meeting and merging of multiple semi-Markov particles interesting, as those might help in investigating stability properties of a diffusion that is modulated by a SMP. See for example \cite{basak1999sankhya} for the stability analysis of Markov modulated diffusion.

\noindent In this paper with the help of an additional parameter, given a semi-Markov law, we consider a family of stochastic flows as appears in \cite{generalsemi}. Then we suitably select a particular flow, for which the investigation of meeting and merging becomes convenient. Although the study of meeting and merging event of a finite-state continuous-time Markov chain is straightforward, that is not the case for semi-Markov counterpart. We show with an example, that the meeting time need not be a merging time for a pair of SMPs. We derive the probability of merging at a meeting time. A set of sufficient conditions are also obtained under which a pair of SMPs eventually merge with probability one.

\noindent The rest of this paper is arranged in the following manner. We present a class of homogeneous SMPs as solution of a system of SDEs in Section 2. A combined process of two solutions having different initial conditions are also introduced here. All the basic assumptions, notations and definitions are stated in this section. In Section 3, we derive an expression of the probability that the combined process meets in the next transition. In Section 4, we address some questions about the event of eventual meeting and merging of combined process. Section 5 contains some concluding remark.
\section{Semi-Markov Flow}
\noindent As the study of non-homogeneous or impure SMP is excluded from this paper, from now we will refer to `pure homogeneous SMP' as semi-Markov process or SMP only. Let $(\Omega,\mathcal{F},P)$ be the underlying probability space and $\mathcal{X}$ the state space, a countable subset of $\mathbb{R}$. Endow the set $\mathcal{X}_{2}:=\{(i,j)\in \mathcal{X}^{2}\ \mid  \ i\neq j\}$ with a total order $\prec$. Let $\mathcal{B}(\mathbb{R}^{d})$ denote the Borel $\sigma$-algebra on $\mathbb{R}^{d}$ and $m_{d}$ denote the Lebesgue measure on $\mathbb{R}^{d}.$ Let $\lambda:=(\lambda_{ij})$ denote a matrix in which the $i\textsuperscript{th}$ diagonal element is $\lambda_{ii}(y) :=- \sum_{j\in\mathcal{X}\setminus\{i\}}\lambda_{ij}(y)$
and for each $(i,j)\in \mathcal{X}_2$, $ \lambda_{ij}\colon[0,\infty)\to(0,\infty)$ is a bounded measurable function such that
\begin{enumerate}
\item[{\bf (A1)}] $C:=\sum_{i\in\mathcal{X}}\sum_{j\in\mathcal{X}\setminus\{i\}}\|\lambda_{ij}\|_{\infty}< \infty $, and
\item[{\bf (A2)}] $\lim_{y\rightarrow\infty}\gamma_{i}(y) =\infty$, where $ \gamma_{i}(y):=\int_{0}^{y}\lambda_{i}(y')dy'$,  where
$\lambda_i(y):=|\lambda_{ii}(y)|$.
\end{enumerate}

\noindent For each $(i,j)\in \mathcal{X}_{2},$ we consider another measurable function $\tilde{\lambda}_{ij}\colon[0,\infty)\to (0,\infty)$ and a collection of generic intervals such that $\tilde{\lambda}_{ij}(y)\leq \|\lambda_{ij}\|_{\infty} $ for almost every $y\ge 0$ and
\begin{align} \label{a}
\lambda_{ij}(y)\leq\tilde{\lambda}_{ij}(y), \textrm{ and }
\Lambda_{ij}(y)=\left(\sum_{(i',j')\prec(i,j)}\tilde{\lambda}_{i'j'}(y)\right)\ +\Big[0,\lambda_{ij}(y)\Big)
\end{align}
for each $y\ge 0$, where $a+B=\{a+b\mid b\in B\}$ for $a\in \mathbb{R}, B\subset \mathbb{R}$. From \eqref{a}, it is clear that for every $y\ge 0$, $\{\Lambda_{ij}(y)\colon (i,j)\in \mathcal{X}_{2}\}$ is a collection of disjoint intervals which is denoted by $\Lambda$.
Following \cite{generalsemi}, we define  $h_{\Lambda}$ and $g_{\Lambda}$ on $ \mathcal{X}\times \mathbb{R}_{+} \times \mathbb{R}$ as
	 \begin{align}
	     h_{\Lambda}(i, y, v):&=\sum_{j\in \mathcal{X}\setminus\{i\}}\!(j-i)\mathds{1}_{\Lambda_{ij}(y)}(v) \label{h}\\
	    g_{\Lambda}(i, y, v):&=y\sum_{j\in \mathcal{X}\setminus\{i\}}\!\mathds{1}_{\Lambda_{ij}(y)}(v) \label{g}
	   \end{align}
where $\mathbb{R}_{+}$ denotes the set of non-negative real numbers. Here $\mathds{1}_B$ is the indicator function of the set $B$. We consider the following system of stochastic differential equations in $X$ and $Y$
	   \begin{equation} \label{eqX}
	     X_{t}=X_{0}+\displaystyle \int_{0^{+}}^{t}\int_{\mathbb{R}}h_{\Lambda}(X_{u-}, Y_{u-}, v)\wp(du,dv)
	   	\end{equation}
	   	\begin{equation}\label{eqY}   	    Y_{t}=Y_{0}+t-\displaystyle \int_{0^{+}}^{t}\int_{\mathbb{R}}g_{\Lambda}(X_{u-}, Y_{u-},v)\wp(du, dv)
	   	\end{equation}
for $t>0$, where the domain of integration $\int_{0+}^{t}$ is $(0,t]$, and the PRM $\wp(du,\ dv)$ is on $\mathbb{R}_{+}\times \mathbb{R}$ with intensity $m_{2}(du,dv)$, and defined on the probability space $(\Omega,\mathcal{F},P).$ We also assume that $\{\wp((0,t] \times dv)\}_{t\geq 0}$ is adapted to $\{\mathcal{F}_{t}\}_{t\geq 0},$ a filtration of $\mathcal{F}$ satisfying the usual hypothesis. Evidently, $\tilde{\lambda}$ controls the left end points of the intervals in $\Lambda$ and so can be utilized to regulate relation between solutions to \eqref{eqX}-\eqref{eqY} with  different initial conditions. Indeed a specific choice namely, $\tilde{\lambda}_{ij}=\|\lambda_{ij}\|_{\infty}$ a.e. simplifies the relation between the intervals $\Lambda_{ij}(y)$ with different values of $i$, $j$ and $y$.
In a more general settings Theorem 2.2 and 2.4 of \cite{generalsemi}, assert that the system \eqref{eqX}-\eqref{eqY} has a unique strong solution $(X,Y) =\{(X_{t}, Y_t)\}_{t\ge 0}$. Also, $X$ is semi-Markov (see Definition 1.1. in \cite{generalsemi}) with transition rate $\lambda$, $Y$ is age, and the embedded chain $\{X_{T_n}\}_{n\ge 0}$ is Markov (see Theorem 3.1 \cite{generalsemi}). Furthermore, $(X,Y)$ is strong Markov which is asserted below.
\begin{theo}\label{theo2.1}
Let $Z= (X,Y)=\{(X_{t},\ Y_{t})\}_{t\geq 0}$ be the unique strong solution to  \eqref{eqX}-\eqref{eqY}. Then the process $Z$ is a strong Markov process.
\end{theo}
\begin{proof}
We note that \eqref{h} and \eqref{g} imply that for each $i$ and almost every $y$, $h_{\Lambda}$ and $g_{\Lambda}$ are sums of functions which are non-zero only on the intervals $\Lambda_{ij}(y)$ for $j\in\mathcal{X}\setminus\{i\}$. Furthermore, $\Lambda_{ij}(y)$ is contained in $\left[0,\sum_{\mathcal{X}_2}\|\lambda_{ij}\|_{\infty}\right]$ for each $i,j$ and almost every $y$. Hence the support of integrand in $v$ variable is contained in $\left[0,\sum_{\mathcal{X}_2}\|\lambda_{ij}\|_{\infty}\right]$ which is a finite interval by (A1). Moreover, the only other condition required for applying Theorem IX.3.9 of \cite{cinlar2011probability} (p-475), to (\eqref{eqX}-\eqref{eqY}) is the Lipschitz condition on the diffusion coefficient, which is zero in this case. Hence the process $Z$ is strong Markov using Theorem IX.3.9 of \cite{cinlar2011probability}.
\end{proof}

\begin{notation}\label{N4.1}
Fix $i,j\in \mathcal{X}$ and $y_1,y_2\geq 0$. Let $Z^1=(X^1,Y^1)$ and $Z^2=(X^2,Y^2)$ be the strong solutions of \eqref{eqX}-\eqref{eqY} with initial conditions
\begin{align*}
    X^1_0=i,Y^1_0=y_1, \textrm{ and }      X_0^2=j,Y_0^2=y_2
\end{align*} respectively.
The jump times of $Z:=(Z^1,Z^2)$ is denoted by $\{T_n\}_{n\ge 1}$ and given by $T_0:=0$ and $T_n := \inf\{ t>T_{n-1} \colon t\in  T^1 \cup T^2\}$ for all $n\ge 1$ where $T^l$ denotes the collection of transition times of $X^l$ for each $l=1,2$.
\end{notation}
\noindent  The above notation is adopted henceforth. We impose the following  restriction on $\tilde{\lambda}$. \begin{enumerate}
\item[{\bf (A3)}] For all $(i,j) \in \mathcal{X}_2$, and for almost every $y\ge 0,$ set $\tilde{\lambda}_{ij}(y)=\|\lambda_{ij}\|_{\infty}$.
\end{enumerate}

\begin{remark}
It is evident that the law of $(X,Y)$ does not depend on the choice of $\tilde{\lambda}$ and depends only on the $\lambda$ matrix and the initial position. Hence, (A3) imposes no condition on the laws of $Z^1$ and $Z^2$ separately. However, the law of $Z$ depends on the choice of $\tilde{\lambda}$. Therefore, (A3) selects a specific flow from the family specified in \eqref{a}. We select that, as the absence of (A3) significantly complicates the relations between the intervals $\Lambda_{ij}(y)$ with different values of $i$, $j$ and $y$ and thus ramifies the relation between $Z_1$ and $Z_2$. On the other hand (A3) implies a very simple relation, namely $\cup_{y\ge 0}\Lambda_{ij}(y)$ are disjoint for different values of $i$ and $j$. This helps us to compute expressions of various probabilities related to  meeting and merging times of $X^1$ and $X^2$. This assumption is central for our study.
\end{remark}
\begin{defi}\label{meet_defi}
Let $T$ be an $\{\mathcal{F}_t\}_{t\ge 0}$ stopping time. The time $\tau_T$ of next meeting by the  processes $X^1$ and $X^2$ after $T$ is given by $\tau_T:= \inf\{ t >T: X_{t}^1=X_{t}^2, \min(Y^1_{t}, Y^2_t)=0\}$. We say that $X^1$ and $X^2$ \textbf{meet eventually} if $\{\tau_0 < \infty\}$ occurs. The random time $\tau':=\inf\{t'\ge 0 \mid  X_{t}^1=X_{t}^2,\forall t\geq t'\}$ is called a merging time of $X^1$ and $X^2$. They are said to \textbf{merge} if $\{\tau'<\infty\}$ occurs. Note that $\tau'$, as defined here, is not necessarily a stopping time.
\end{defi}
\noindent The nature of meeting and merging for a semi-Markov family is more involved than those for the Markovian special case. We clarify this in the next section.

\section{Meeting and Merging at the Next Transition}
\noindent Markov pure jump processes, although form a subclass of \eqref{eqX}-\eqref{eqY}, deserve a separate mention due to its importance. Hence we first consider a special case where $\lambda$ is independent of the age variable $y$ and satisfies (A1). Evidently, (A2) holds too. Furthermore, (A3) implies that $\tilde{\lambda}_{ij}(y)=\lambda_{ij}$, a constant function for each $(i,j)\in \mathcal{X}_2$. Hence \eqref{eqX} reduces to
\begin{align}\label{markoveq}
    X_t=X_0+\int_{0^+}^t \tilde{h}(X_{s-},v)\wp(ds,dv)
\end{align}
where $\tilde{h}(i,v):=h_{\Lambda}(i,y,v)=\sum_{j\in \mathcal{X}\setminus\{i\}}(j-i)\mathds{1}_{\Lambda_{ij}(y)}(v)$ is constant in $y$, as the intervals $\Lambda_{ij}(y)$, do not vary with $y$ variable. Uniqueness result of \eqref{markoveq} implies the following.
\begin{theo} \label{theo4.1}
Let $X^1$ and $X^2$ be strong solutions of SDE \eqref{markoveq} with initial states $X_0^1=i$ and $X_0^2=j$ respectively. Then, if $X^1$ and $X^2$ meet, they merge at the first meeting. \end{theo}
\begin{proof} For a $\omega\in \Omega,$ if there exists a $t'>0$ such that $X_{t'}^1(\omega)=X_{t'}^2(\omega)=k$ for some $k\in \mathcal{X}$\footnote{ Note that, it is not necessary that $X^1$ and $X^2$ transit to state $k$ at the same time.}, then using \eqref{markoveq}
for $t>t',$ both $X^1(\omega)$ and $X^2(\omega)$ solve
\begin{align}
  \nonumber  X_t(\omega)&=X_{t'}(\omega)+\int_{t'^{+}}^t \tilde{h}(X_{s-}(\omega),v)\wp(ds,dv) (\omega)= k +\int_{t'^{+}}^t \tilde{h}(X_{s-}(\omega),v)\wp(ds,dv)(\omega).
\end{align}
Now using almost sure uniqueness of the strong solution of the above SDE, $X^1(\omega)$ and $X^2(\omega)$ would be identical from time $t'$ onward. Thus $X^1$ and $X^2$ merge at their first meeting time. \qed
\end{proof}
\noindent It is interesting to note that, if $\lambda$ is constant, the merging time of $X^1$ and $X^2$, as given in Theorem \ref{theo4.1}, is a stopping time. This is because, merging and meeting times coincide, and the latter is a stopping time. This consequence is not valid for a general semi-Markov family. Indeed, if $X^1$ and $X^2$ are as in Notation \ref{N4.1}, at the meeting time they may have unequal ages and those age variables appear in the SDE \eqref{eqX}-\eqref{eqY}. So, the mere uniqueness of the SDE does not imply merging at the first meeting time. We produce below an example of a meeting event which is not the merging of a semi-Markov family.
\begin{ex}\label{ex}
Let $\mathcal{X}=\{1,2\}$, with $(1,2)\prec(2,1)$; also $\lambda_{12}(y)= \lambda_{21}(y)= \frac{y}{1+y}$, and $\tilde{\lambda}_{12}(y)=\tilde{\lambda}_{21}(y)= \sup_{(0,\infty)}|\frac{y}{1+y}|=1$ for all $y\ge 0$. Thus for every $(i,j)\in \mathcal{X}_2$, $\Lambda_{ij}(y)=[i-1, i-1 + \frac{y}{1+y})$. We further assume that $Z^l=(X^l,Y^l)$ is the strong solution of \eqref{eqX}-\eqref{eqY} with above parameters and initial conditions $(X^l_0, Y^l_0)= (l,\mathds{1}_{\{2\}}(l))$ for $l=1,2$ respectively. Now fix a sample $\omega \in \Omega$ such that $\wp(\omega)|_{[0,3/2]\times [0,2]} = \delta_{(1,3/2)}+\delta_{(3/2,1/2)}$, the addition of two Dirac measures at $(1,3/2)$ and $(3/2,1/2)$ respectively.
\begin{figure}[ht]
\centering
\includegraphics[width=0.6\linewidth]{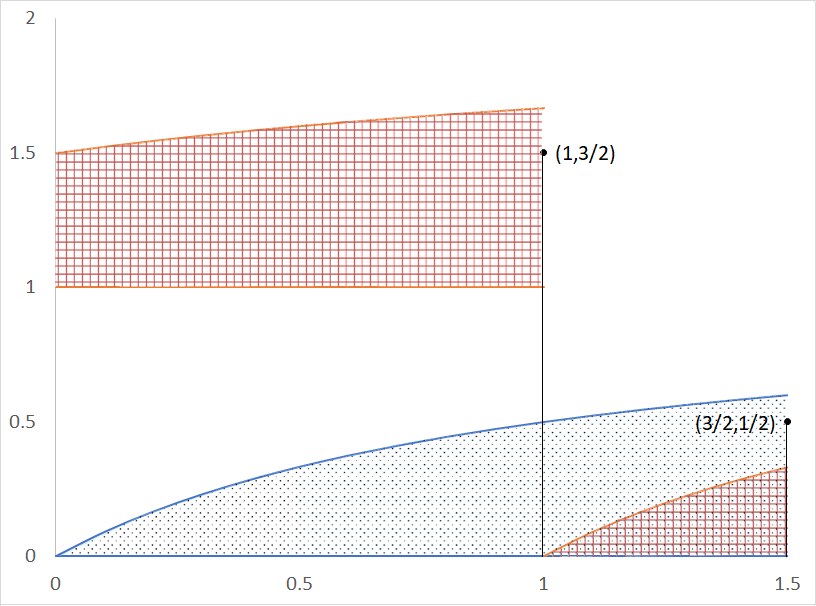} \caption{The $t$ and $v$ variables are plotted along horizontal and vertical axes. The point masses are shown by black dots. The intervals relevant for transitions of the first and second processes are plotted vertically and shown in blue and red respectively.}\label{exfig} \end{figure}
Then none of the processes has transition until time $t=1$. Hence, for both $l=1,2$,
\begin{align*}
X^l_{1-}(\omega) = l, \, \textrm{ and, }\, Y^l_{1-} (\omega)= Y^l_{0}+1- \int_{(0,1)}\int_{\mathbb{R}}g_{\Lambda}(X^l_{u-}(\omega), Y^l_{u-}(\omega),v)\wp(du, dv)(\omega)= \mathds{1}_{\{2\}}(l) +1 =l.
\end{align*}
Then from \eqref{eqX}
\begin{align*}
X^l_{1}(\omega) =& X^l_{1-} (\omega) +\int_{\mathbb{R}}h_{\Lambda}(X^l_{1-}(\omega), Y^l_{1-}(\omega), v)\wp(\{1\}, dv)(\omega)=  l + h_{\Lambda}(l,l,3/2).
\end{align*}
Therefore, using \eqref{h} and the intervals $\Lambda_{12}(1), \Lambda_{21}(2)$, we get $X^1_1(\omega)=1+ (2-1)\mathds{1}_{[0, 1/2)}(3/2)= 1$ and $X^2_1(\omega)= 2 + (1-2)\mathds{1}_{[1, 1+ 2/3)}(3/2)= 1$.
Thus, $t=1$ is a meeting time. However, this is not a merging time, because at $t=3/2$, $X^1(\omega)$ and $X^2(\omega)$ separate, which is shown below. We note that until $t=3/2$, $X^1(\omega)$ and $X^2(\omega)$ are at state 1 since $t=0$, and $t=1$ respectively. So, while the pre-transition state $X^l_{3/2-}(\omega)$ is $1$ for each $l=1,2$, the pre-transition ages  $Y^1_{3/2-}(\omega)$, and $Y^2_{3/2-}(\omega)$ are $3/2$ and $1/2$ respectively. Consequently,
\begin{align*}
X^l_{3/2}(\omega) =& 1 +\int_{\mathbb{R}} h_{\Lambda}(1, Y^l_{3/2-}(\omega),v)\wp(\{3/2\}, dv)(\omega)
= 1+ \mathds{1}_{\Lambda_{1 2}(Y^l_{3/2-}(\omega))}(1/2)
= \begin{cases}
2, & \mbox{for } l=1 \\
1, & \mbox{for } l=2
\end{cases}
\end{align*}
since, $1/2 \in \Lambda_{1 2}(3/2)=[0,\frac{3/2}{1+3/2})=[0,3/5)$ and $1/2 \notin \Lambda_{1 2}(1/2)=[0,\frac{1/2}{1+1/2})=[0,1/3)$.
\end{ex}

\begin{theo}\label{theoSM_meet} Assume (A3). Let $Z^1=(X^1,Y^1)$ and $Z^2=(X^2,Y^2)$ be as in Notation \ref{N4.1} where $i\neq j$. The probability of $X^1$ and $X^2$ meeting in the next transition is
$$\int_0^\infty e^{-\int_0^y(\lambda_i(y_1+t)+\lambda_j(y_2+t))dt}(\lambda_{ij}(y_1+y)+\lambda_{ji}(y_2+y))dy.$$
\end{theo}
\proof In this proof we will utilise that for every $y'\ge 0$ and $y''\ge 0$,  $\cup_{k\neq i}\Lambda_{ik}(y')$ is disjoint to $\cup_{k\neq j}\Lambda_{jk}(y'')$ when $i\neq j$. This is consequence of definitions of the intervals in \eqref{a}, and (A3). Non-meeting event in the next transition of $X^1$ and $X^2,$ happens in two ways.\\
Case 1: $X^1$ has the first transition to a state which is different from $X^2_0$ before $X^2$ transits for the first time. This event can be written as $\mathcal{E}:=\{X^1_{T_1-}\neq X^1_{T_1},X^1_{T_1}\neq X^2_{T_1}\}$. We will make use of $ P(\mathcal{E}\mid \mathcal{F}_{0})=E[P(\mathcal{E}\mid T_1)\mid \mathcal{F}_0]$, and the expression of conditional density $\eta_{T_1}$ of $T_1$ given $X^1_{T_0}=i,X^2_{T_0}=j,Y^1_{T_0}=y_1,Y^2_{T_0}=y_2$. Clearly, $P(\mathcal{E} |T_1=y)$ is $\frac{m_1( \underset{k\notin \{i,j\}}{\cup} \Lambda_{ik}(y_1+y))}{m_1( \underset{k\notin \{i\}}{\cup} \Lambda_{ik}(y_1+y) \cup \underset{k\notin \{j\}}{\cup} \Lambda_{jk}(y_2+y))}$
$=\frac{\lambda_i(y_1+y)-\lambda_{ij}(y_1+y)}{\lambda_{i}(y_1+y)+\lambda_j(y_2+y)}$. Moreover, $\eta_{T_1}(y)=e^{-m_2(B)}(\lambda_i(y_1+y)+\lambda_j(y_2+y))$, where $B:= \underset{t\in[0,y)} {\cup} \left(\{t\} \times \left((\underset{k\notin \{i\}}{\cup} \Lambda_{ik}(y_1+t)) \cup(\underset{k\notin \{j\}}{\cup} \Lambda_{jk} (y_2+t))\right)\right).$ Indeed, the event of no transition of $X^1$ and $X^2$ until first $y$ unit of time, is equivalent to $\{\wp(B)=0\}$,
the non-occurrence of Poisson point mass in $B$.
Clearly, $P(\{\wp(B)=0\}\mid X^1_{T_0}=i,X^2_{T_0}=j,Y^1_{T_0}=y_1,Y^2_{T_0}=y_2)$ is equal to $e^{-m_2(B)}$, and $m_2(B) =\int_0^y(\lambda_i(y_1+t)+\lambda_j(y_2+t))dt$. Hence
\begin{align}
P(\mathcal{E}\mid \mathcal{F}_0)=&\int_{0}^{\infty} P(\mathcal{E}\mid T_1=y)\eta_{T_1}(y)dy\nonumber \\
=&\int_{0}^{\infty} e^{-\int_0^y(\lambda_i(y_1+t)+\lambda_j(y_2+t))dt}[\lambda_i(y_1+y)-\lambda_{ij}(y_1+y)]dy. \label{part_1}
\end{align}
Similarly for Case 2, i.e., $X^2$ has the first transition to a state, different from $X^1_0$, before $X^1$ transits for the first time is given by,
\begin{align}
    P(X^2_{T_1-}\neq X^2_{T_1},X^2_{T_1}\neq X^1_{T_1}\mid \mathcal{F}_0)
    &=\int_0^\infty e^{-\int_0^y(\lambda_i(y_1+t)+\lambda_j(y_2+t))dt}[\lambda_j(y_2+y)-\lambda_{ji}(y_2+y)]dy\label{part_2}.
\end{align}
Hence the total probability (denoted by $a'_{(i,j,y_1,y_2)}$) of not meeting in the next transition is sum of the probabilities appearing in  \eqref{part_1}, and \eqref{part_2}. \\Using $\phi_1(y) := e^{-\int_0^y(\lambda_i(y_1+t)+\lambda_j(y_2+t))dt}\left(\lambda_i(y_1+y)+\lambda_j(y_2+y)\right)$,
\begin{align}\label{total2}
    a'_{(i,j,y_1,y_2)}&=\int_0^\infty\left( \phi_1(y) - e^{-\int_0^y(\lambda_i(y_1+t)+\lambda_j(y_2+t))dt}\left[(\lambda_{ij}(y_1+y)+\lambda_{ji}(y_2+y))\right]\right)dy \\
  \nonumber  &=1-\int_0^\infty e^{-\int_0^y(\lambda_i(y_1+t)+\lambda_j(y_2+t))dt}{(\lambda_{ij}(y_1+y)+\lambda_{ji}(y_2+y))}dy
\end{align}
as $\int_0^\infty \phi_1(y) dy =1$. Hence $1-a'_{(i,j,y_1,y_2)}$, the probability of meeting of $X^1$ and $X^2$ in the next transition has the desired expression.\qed
\begin{defi}
Let $Z^1=(X^1,Y^1)$ and $Z^2=(X^2,Y^2)$ be the strong solutions of \eqref{eqX}-\eqref{eqY} with two different sets of initial conditions. Let $\mathcal{P}(k,y)$ denote the regular conditional probability of merging of $X^1$ and $X^2$ at a meeting time given meeting occurred in finite time,  $k$ is the meeting state, and $y$ is the age of the chain which arrives at $k$ prior to the meeting time.
\end{defi}
\begin{theo}\label{theo4.9}
Assume (A3). Then
   \begin{align}
        \mathcal{P}(k,y)=\int_{0}^\infty e^{-\int_{0}^{y'}\underset{k'\neq k}{\sum}(\lambda_{kk'}(t)\vee\lambda_{kk'}(y+t))dt}\left[\sum_{k'\neq k}\lambda_{kk'}(y')\wedge\lambda_{kk'}(y+y')\right]dy'
   \end{align}
where $a\vee b =\max(a,b)$ and $a\wedge b =\min(a,b)$.
\end{theo}
\proof Let $t'$ denote a meeting time of
$X^1$ and $X^2$. It is given that $t'$ is finite, with $X^1_{t'}=k =X^2_{t'}$,  $Y^1_{t'} \wedge Y^2_{t'}=0$ and $Y^1_{t'}\vee Y^2_{t'}=y>0$.
Let $\vartheta$ denote the duration both the processes stay at $k$ before either of them transit to some other state.
Clearly, the event of no transition of $X^1$ and $X^2$ for next $y'$ unit of time after $t'$, is equivalent to the event where no Poisson point mass belongs to the set $B:= \underset{t\in[0,y')} {\cup} \left(\{t'+t\} \times \underset{k'\notin \{k\}}{\cup} (\Lambda_{kk'}(t) \cup\Lambda_{kk'} (y+t))\right)$.
Evidently, this event occurs with probability $e^{-m_2(B)}$, as $m_2$ is the intensity of the Poisson random measure.
Since, simultaneous occurrence of this event and the event of a Poisson point mass lying on the line segment $\{t'+y'\} \times \underset{k'\notin \{k\}}{\cup} (\Lambda_{kk'}(y') \cup\Lambda_{kk'} (y+y'))$ is equivalent to the occurrence of $\{\vartheta =y'\}$, the expression of conditional density $\eta_{\vartheta}$ of $\vartheta$ is given by $\eta_{\vartheta}(y')=e^{-m_2(B)}m_1(\underset{k'\notin \{k\}}{\cup} A_{k'}(y'))$, where $A_{k'}(y'):= \Lambda_{kk'}(y')\cup  \Lambda_{kk'}(y+y')$ for every $k'\neq k$.
As $\tilde{ \lambda}_{i'j'}(y)$ is set as constant $\|\lambda_{i'j'}\|_{\infty}$ for almost every $y$ (Assumption (A3)), due to the definitions of the intervals in \eqref{a}, for almost every $y\ge 0$ and $y'\ge 0$ the collection $\{A_{k'}(y')\}_{k'\in \mathcal{X}\setminus\{k\}}$ is disjoint. Moreover, due to (A3) the left end points of the intervals  $\Lambda_{kk'}(y')$ and $\Lambda_{kk'}(y+y')$ are common (see \eqref{a}). Thus the Lebesgue measures of $\Lambda_{kk'}(y') \cup\Lambda_{kk'} (y+y')$ and $\Lambda_{kk'}(y') \cap\Lambda_{kk'} (y+y')$ are $\lambda_{kk'}(y')\vee\lambda_{kk'}(y+y')$ and $\lambda_{kk'}(y')\wedge\lambda_{kk'}(y+y')$ respectively. Thus $\eta_{\vartheta}(y')=\exp({-\int_{0}^{y'}\underset{k'\neq k}{\sum}(\lambda_{kk'}(t)\vee\lambda_{kk'}(y+t))dt}) \sum_{k'\neq k}\lambda_{kk'}(y')\vee\lambda_{kk'}(y+y').$

We consider two cases regarding the transition of $X^1$ and $X^2$, at $t' + \vartheta$ which are (i) simultaneous, and (ii) non-simultaneous. Case (ii) implies that $X^1$ and $X^2$ will depart in the next transition. So, under case (ii), $t'$ is not a merging time. Consequently, case (i) is necessary for $t'$ to be the merging time. We show that case (i) is a sufficient condition too. We recall that at $t'+\vartheta$ the Poisson point mass (which is responsible for the transition) lies in only one of the members of the disjoint family $\{A_{k'}(\vartheta)\}_{k'}$ with probability one. Therefore, under case (i), at $t'+\vartheta$, $X^1$ and $X^2$ enter into an identical state and the ages $Y^1$ and $Y^2$ become zero and therefore, the uniqueness of the SDE \eqref{eqX}-\eqref{eqY} implies merging at time $t'$.

For case (i) to occur, the point mass must lie in $\{t'+\vartheta\} \times \cup_{k'\notin \{k\}} \Lambda_{kk'}(\vartheta)\cap  \Lambda_{kk'}(y+\vartheta)$. On the other hand if the point mass lies in $\{t'+\vartheta\} \times  \cup_{k'\notin \{k\}} \left( A_{k'}(\vartheta)\setminus \Lambda_{kk'}(\vartheta)\cap  \Lambda_{kk'}(y+\vartheta)\right)$, then the transition at $t'+\vartheta$ is of case (ii).
Hence, for almost every $y$ and $y'$, the conditional probability of merging given $\vartheta=y'$
is equal to $\frac{\sum_{k'\neq k}\lambda_{kk'}(y')\wedge\lambda_{kk'}(y+y')}{\underset{k'\neq k}{\sum} (\lambda_{kk'}(y')\vee\lambda_{kk'}(y+y'))}$.
Thus $P\left( X^1_t=X^2_t,\forall t\geq t' \mid \{t'<\infty\}, X^1_{t'} = X^2_{t'} = k, Y^1_{t'} \wedge Y^2_{t'}=0, Y^1_{t'}\vee Y^2_{t'}=y\right)$ is equal to
\begin{align*}
&\int_{0}^{\infty}P\left(X^1_t=X^2_t,\forall t\geq t' \mid  \{t'<\infty\}, X^1_{t'} = X^2_{t'} = k, Y^1_{t'} \wedge Y^2_{t'}=0, Y^1_{t'}\vee Y^2_{t'}=y, \vartheta=y'\right)\eta_{\vartheta}(y')dy'\\
    &=\int_{0}^{\infty} e^{-\int_{0}^{y'}\underset{k'\neq k}{\sum}(\lambda_{kk'}(t)\vee\lambda_{kk'}(y+t))dt}\left(\sum_{k'\neq k}\lambda_{kk'}(y')\wedge\lambda_{kk'}(y+y')\right)dy'.
\end{align*}
\noindent
 This completes the proof. \qed
\begin{rem}
It is interesting to note that for Markov special case, where the transition rate matrix $\lambda$ is independent of the age variable $y$, a direct calculation gives that $\mathcal{P}(k,y) =1$. This makes Theorem \ref{theo4.1}, a corollary of the above theorem. On the other hand by considering the two-state semi-Markov chain given in Example \ref{ex}, one can obtain for each $k=1,2$, $\lim_{y\to \infty} \mathcal{P}(k,y) =\int_0^\infty e^{-y'}\frac{y'}{1+y'}dy'< \frac{1}{2} \int_0^1 e^{-y'} dy'+\int_1^\infty e^{-y'}dy'= \frac{1+e^{-1}}{2}<1$. This further clarifies that a meeting time for the flow in Example \ref{ex} need not be a merging time. Below we show that the chance of merging for a general semi-Markov chain increases to 1 as $y$ decreases to zero, provided that the transition rate is continuous at zero.
\end{rem}

\begin{prop}
Assume (A3) and that $y\mapsto \lambda(y)$ is continuous at zero.  As $y$ tends to zero, $\mathcal{P}(k,y)$ converges to $1$.
\end{prop}
\begin{proof}
 From Theorem \ref{theo4.9}
\begin{align}
  \nonumber  \lim_{y\to 0}\mathcal{P}(k,y)&=\lim_{y\to 0}\int_0^{\infty}e^{-\int_{0}^{y'}\underset{k'\neq k}{\sum}(\lambda_{kk'}(t)\vee\lambda_{kk'}(y+t))dt}\left(\sum_{k'\neq k}\lambda_{kk'}(y')\wedge\lambda_{kk'}(y+y')\right)dy'.
\end{align}
Due to the continuity of $\lambda$ in $y$, the integrand converges pointwise to
$$\psi(y'):= \sum_{k'\neq k}\lambda_{kk'}(y') \exp\left({-\int_{0}^{y'} {\sum_{k'\neq k}}\lambda_{kk'}(t)dt}\right).$$
The integrand is also uniformly dominated by $\psi$, which is integrable on $[0,\infty)$. Indeed $\int_{[0,\infty)}\psi(y') dy'= 1$. Thus the result follows using dominated convergence theorem. \qed
\end{proof}

\section{Eventual Meeting and Merging}
It is important to note that the strict positivity of entries of the rate matrix, as assumed in this paper, implies irreducibility of the process. It is also known that mere irreducibility of a Markov chain does not ensure the convergence. However, the meeting event of two chains may take place even if the chains do not converge. The discrete time Markov chain on two states having zero probability of transition to the same state constitutes an example where chains with different initial states never meet due to its periodicity. Nevertheless, the same phenomena is untrue for its continuous time version. Indeed, if two such chains (Markov/semi-Markov), having bounded transition rate and driven by the same noise (the Poisson random measure) start from two different states, they meet surely at the next transition. In this paper, due to the consideration of processes having bounded transition rate, the discrete time scenario is excluded. Thus an ergodicity assumption is not needed for assuring eventual meeting. The next theorem establishes eventual meeting of Markov special case under finiteness assumption of the state space.
\begin{theo}\label{theoM_meet}
Let $X^1$ and $X^2$ be as in Theorem \ref{theo4.1}.
\begin{enumerate}
    \item The conditional probability of meeting in the next transition given $\mathcal{F}_0$ is $\frac{(\lambda_{ij}+\lambda_{ji})}{\lambda_{i}+\lambda_{j}}$.
    \item If $\mathcal{X}$ is finite,  $X^1$ and $X^2$ eventually meet with probability 1.
\end{enumerate}
\end{theo}
\proof
Recall $X^1_0=i$, $X^2_0=j$, and the sequence $\{T_n\}$ from Notation \ref{N4.1}. By applying Theorem \ref{theoSM_meet} for the Markov special case, we can write the conditional probability of meeting in the next transition of $X^1$ and $X^2$, given the initial conditions as
\begin{align*}
\int_0^{\infty}e^{-\int_0^y (\lambda_{i}+\lambda_{j})dt} (\lambda_{ij}+\lambda_{ji})dy = \frac{(\lambda_{ij} + \lambda_{ji})} {(\lambda_{i}+\lambda_{j})} \int_0^{\infty}e^{-(\lambda_{i}+\lambda_{j})y} (\lambda_{i}+\lambda_{j})dy= \frac{(\lambda_{ij} + \lambda_{ji})} {(\lambda_{i}+\lambda_{j})}.
\end{align*}
Hence the part (1) is proved. Since, $\lambda_{ij}>0$ for all $i\neq j$ and $\mathcal{X}$ is finite, $\min_{i,j}\frac{(\lambda_{ij}+\lambda_{ji})}{\lambda_{i}+\lambda_{j}}>0$. Thus $\max_{i,j} a_{(i,j)}<1$ where $a_{(i,j)}$ denotes the probability of not meeting in the next transition. Now since $\{T_n\}$ is a sequence of stopping times, using Theorem \ref{theo2.1}, we get
\begin{align}\label{aij}
E\left[ \mathds{1}_{\{X^1_{T_n}\neq X^2_{T_n}\}}\mid \mathcal{F}_{T_{n-1}}\right]= a_{(X^1_{T_{n-1}},X^2_{T_{n-1}})} \le \max_{i,j} a_{(i,j)}<1.
\end{align}
\noindent The event of never meeting of processes $X^1$ and $X^2$ is identical to the repeated occurrence of $\{X^1_{T_n}\neq X^2_{T_n}\}$ for all $n\ge 1$.
Hence, using the fact (thanks to (A2)) that the chains experience infinitely many transitions with probability 1, the probability of never meeting, $P\left(X^1_t\neq X^2_t,\forall t\ge 0\mid \mathcal{F}_{0}\right)$ matches with $\lim_{N\to \infty}P\left(\cap_{n=1}^{N} \{X^1_{T_n}\neq X^2_{T_n}\}\mid \mathcal{F}_{0}\right)$. Next if
\begin{align}\label{repeat}
E\left[ \prod_{n=1}^N \mathds{1}_{\{X^1_{T_n}\neq X^2_{T_n}\}}\mid \mathcal{F}_{0}\right]
\le \max_{i,j} a_{(i,j)}
E\left[ \prod_{n=1}^{N-1} \mathds{1}_{\{X^1_{T_n}\neq X^2_{T_n}\}}\mid \mathcal{F}_{0}\right],
\end{align}
holds for all $N\ge 1$, using that repeatedly, we get
$$ P\left(\cap_{n=1}^{N} \{X^1_{T_n}\neq X^2_{T_n}\}\mid \mathcal{F}_{0}\right) = E\left[ \prod_{n=1}^N \mathds{1}_{\{X^1_{T_n}\neq X^2_{T_n}\}}\mid \mathcal{F}_{0}\right] \le \left(\max_{i,j} a_{(i,j)}\right)^N
$$
for all $N\ge 1$. The right side clearly vanishes as $N\to \infty$, and thus $P\left(X^1_t\neq X^2_t,\forall t\ge 0\mid \mathcal{F}_{0}\right)$ is zero as desired, provided \eqref{repeat} holds. Finally \eqref{repeat} is shown using \eqref{aij} below
\begin{align*}
E\left[ \prod_{n=1}^N \mathds{1}_{\{X^1_{T_n}\neq X^2_{T_n}\}}\mid \mathcal{F}_{0}\right]
&= E\left[E\left[ \prod_{n=1}^N \mathds{1}_{\{X^1_{T_n}\neq X^2_{T_n}\}}\mid \mathcal{F}_{T_{N-1}}\right]\mid \mathcal{F}_{0}\right]\\
&= E\left[\left(\prod_{n=1}^{N-1}\mathds{1}_{\{X^1_{T_n}\neq X^2_{T_n}\}}\right)E\left[\mathds{1}_{\{X^1_{T_N}\neq X^2_{T_N}\}}\mid \mathcal{F}_{T_{N-1}}\right]\mid \mathcal{F}_{0}\right]\\
&\le \max_{i,j} a_{(i,j)}
E\left[\prod_{n=1}^{N-1} \mathds{1}_{\{X^1_{T_n}\neq X^2_{T_n}\}}\mid \mathcal{F}_{0}\right]
\end{align*}
for all $N\ge 1$. Hence the proof of part(2) is complete. \qed

\noindent In the above proof, the second part of the theorem has been proved using the first part. However, the former has been proved in Lemma 3.5 of \cite{basak1999sankhya} without utilizing part 1, under identical assumption in a different approach.

\noindent It is important to note that for ensuring almost sure eventual meeting, we have assumed finiteness of $\mathcal{X}$ in the above theorem, whereas in the proof we have used $\min_{i,j}\frac{(\lambda_{ij}+\lambda_{ji})}{\lambda_{i}+\lambda_{j}}>0$ only. In the following lemma we show that under (A1), these conditions are equivalent.
\begin{lem} \label{lem4.2} Let $\lambda$ be a transition rate matrix of a Markov chain obeying (A1). If $\mathcal{X}$ is infinite, $\underset{i,j}{\inf}\frac{\left(\lambda_{ij}+\lambda_{ji}\right)}{\lambda_{i}+\lambda_{j}}$ is zero.
\end{lem}
\proof
Fix a $j\in \mathcal{X}$. Since, due to Assumption (A1), $\sum_{i=1}^\infty\lambda_i<\infty$,   given $\epsilon>0$  there exists an $i_{\epsilon,j}$  such that  $\lambda_i<\epsilon\lambda_j$ $\forall i>i_{\epsilon,j}.$ So we get an inequality $\lambda_{ij}\le \lambda_i<\epsilon\lambda_j$ for all $i>i_{\epsilon,j}$. Using this inequality we have the following relation,
\begin{equation}\label{relation1}
    \frac{\lambda_{ij}+\lambda_{ji}}{\lambda_i+\lambda_j}<\frac{\epsilon\lambda_j+\lambda_{ji}}{\lambda_i+\lambda_j}<\frac{\epsilon\lambda_j+\lambda_{ji}}{\lambda_j}=\epsilon+\frac{\lambda_{ji}}{\lambda_j},
\end{equation}
for all $i>i_{\epsilon,j}$.
For each $j$ we also have  $\lambda_j= \underset{i\in \mathcal{X}\setminus \{j\}}{\sum}\lambda_{ji} <\infty$. Hence, there exists a $i^*_{j,\epsilon}$ such that for all $i>i^*_{j,\epsilon}$ we have $\lambda_{ji}<\epsilon\lambda_j$. Now, using \eqref{relation1}, we get for each $i>\max(i_{\epsilon,j},i^*_{j,\epsilon}),$
\begin{align*}\label{relation3}
    \frac{\lambda_{ij}+\lambda_{ji}}{\lambda_i+\lambda_j}<\epsilon+\frac{\epsilon\lambda_j}{\lambda_j}=2\epsilon.
\end{align*} Since $\epsilon$ is arbitrary, the above implies that for each $j\in \mathcal{X}$,
\begin{equation}\label{limiting_i}
    \lim_{i\to \infty}\frac{\lambda_{ij}+\lambda_{ji}}{\lambda_i+\lambda_j}=0.
\end{equation}
Similarly by interchanging the roles of $i$ and $j$ in the above argument, one obtains
\begin{equation}\label{limiting_j}
    \lim_{j\to \infty}\frac{\lambda_{ij}+\lambda_{ji}}{\lambda_i+\lambda_j}=0
\end{equation}
for each $i\in \mathcal{X}$. Hence from \eqref{limiting_i}, and \eqref{limiting_j}, we conclude $\underset{i,j}{\inf}(\frac{\lambda_{ij}+\lambda_{ji}}{\lambda_i+\lambda_j})=0.$ \qed

\noindent Next we wish to investigate the eventual meeting event for semi-Markov family. Clearly, in view of Theorem \ref{theoM_meet}(2), a condition like $\inf_{(i,j)\in \mathcal{X}_2, y_1, y_2, y}\frac{(\lambda_{ij} (y_1+y)+\lambda_{ji}(y_2+y))}{\lambda_{i}(y_1+y)+\lambda_{j}(y_2+y)} >0$ is needed for this purpose. However, finiteness of $\mathcal{X}$ is not enough to ensure that. We consider the following assumption.
\begin{enumerate}
\item[{\bf (A4)}] $\mathcal{X}$ is finite and  $\sup_{(i,j)\in \mathcal{X}_2, y_1\ge 0, y_2\ge 0}\left\|1-\frac{\left(\lambda_{ij}(y_1+\cdot)+\lambda_{ji}(y_2+\cdot)\right)}{\lambda_{i}(y_1+\cdot)+\lambda_{j}(y_2+\cdot)}\right\|_{L^\infty} < 1.$
\end{enumerate}

\begin{theo}\label{SM_EM}
Assume (A1)-(A4) and that $X^1$ and $X^2$ are as in Notation \ref{N4.1}. Then $X^1$ and $X^2$ eventually meet with probability 1.
\end{theo}
\begin{proof} Using $\phi_2 (y) := 1- \frac{(\lambda_{ij}(y_1+y)+\lambda_{ji}(y_2+y))}{\lambda_{i}(y_1+y)+\lambda_{j}(y_2+y)}$, we rewrite \eqref{total2} as
\begin{align*}
a'_{(i,j,y_1,y_2)}&=\int_0^\infty \phi_1(y)\phi_2(y) dy \le \|\phi_1\|_{L^1}\|\phi_2\|_{L^\infty} = \| \phi_2\|_{L^\infty}.
\end{align*}
Now by a direct application of (A4), we get that supremum of $\|\phi_2\|_{L^\infty}$ over all $(i,j)\in \mathcal{X}_2, y_1\ge 0, y_2\ge 0$ is less than 1, which implies that
\begin{equation}\label{equivalent_assumption_2}
\sup_{(i,j)\in \mathcal{X}_2, y_1\ge 0, y_2\ge 0}  a'_{(i,j,y_1,y_2)}<1.
\end{equation}
Again as in the proof of Theorem \ref{theoM_meet}, the total probability of never meeting is the probability of intersection of occurrence of not meeting in next transition for every transition, and (A2) ensures almost sure infinite transitions. Moreover, since $(Z^1, Z^2)$ is strong Markov (Theorem \ref{theo2.1}) and $\{T_n\}_{n\ge 1}$ are stopping times $P\left(\{X^1_{T_n}\neq X^2_{T_n}\}\mid \mathcal{F}_{T_{n-1}}\right) = a'_{(X^1_{T_{n-1}},X^2_{T_{n-1}},Y^1_{T_{n-1}},Y^2_{T_{n-1}})}$ which is not more than the left side of \eqref{equivalent_assumption_2}. Therefore, in the similar line of the proof of Theorem \ref{theoM_meet}, we get
\begin{align}\label{aijn}
E\left[ \prod_{n=1}^N \mathds{1}_{\{X^1_{T_n}\neq X^2_{T_n}\}}\mid \mathcal{F}_{0}\right] \le \left(\sup_{(i,j)\in \mathcal{X}_2, y_1\ge 0, y_2\ge 0}  a'_{(i,j,y_1,y_2)}\right)^N
\end{align}
and  $P\left(X^1_t\neq X^2_t,\forall t\ge 0\mid \mathcal{F}_{0} \right) = \lim_{N\to \infty}  E\left[ \prod_{n=1}^N \mathds{1}_{\{X^1_{T_n}\neq X^2_{T_n}\}}\mid \mathcal{F}_{0}\right] $. This limit is zero from \eqref{equivalent_assumption_2} and \eqref{aijn}. Thus the probability of never meeting is zero. \qed
\end{proof}
\noindent Under (A1)-(A4), the pair $(X^1, X^2)$ not only surely meet, the expected number of transitions needed for meeting is also finite. A rather stronger result is shown below.
\begin{theo}\label{moment_r}
Assume (A1)-(A4) and that $X^1$ and $X^2$ are as in Notation \ref{N4.1}. If $N$ denotes the number of collective transitions until the first meeting time of  $X^1$ and $X^2$, then $E[N^r]<\infty$ for any $r\ge 1$.
\end{theo}
\begin{proof}
For the sake of brevity, we write $a'_{(Z^1_{T_{n}},Z^2_{T_{n}})}$ for $a'_{(X^1_{T_{n}},X^2_{T_{n}},Y^1_{T_{n}},Y^2_{T_{n}})}$, a notation that appears in the proof of Theorem \ref{theoSM_meet}.
Since $N$ denotes the number of collective transitions until the first meeting time, using the above notation and \eqref{aijn}, we get for all $n\ge 0$
\begin{align*}
P(N=n+1) =&  E\left[\left(\prod_{r=1}^n \mathds{1}_{\{X^1_{T_r}\neq X^2_{T_r}\}} \right) E\left( \mathds{1}_{\{X^1_{T_{n+1}} = X^2_{T_{n+1}}\}} \mid \mathcal{F}_{T_n} \right)\right] \\
\le& \left(\sup_{z^1,z^2} a'_{(z^1, z^2)}\right)^n \left(1- \inf_{z^1,z^2} a'_{(z^1, z^2)}\right),
\end{align*}
by following the convention that product and intersection of an empty family are 1 and empty set respectively.
Thus the $r^{th}$ raw moment, $E [N^r]$ is
\begin{align*}
\sum_{n=1}^\infty n^r P(N=n)
    \le & \left(1- \inf_{z^1,z^2} a'_{(z^1, z^2)}\right)\left(\sup_{z^1,z^2} a'_{(z^1, z^2)}\right)^{-1} \sum_{n=1}^\infty n^r \left(\sup_{z^1,z^2} a'_{(z^1, z^2)}\right)^{n}.
\end{align*}
The infinite series on the right converges provided $\sup_{z^1,z^2} a'_{(z^1, z^2)} <1$ which is ensured in \eqref{equivalent_assumption_2} due to the assumption (A4). To be more precise, that series sum is expressed as $Li_{-r}(\sup_{z^1,z^2} a'_{(z^1, z^2)})$ where $Li_r(z)$ is polylogarithm function of order $r$ and with argument $z$. Thus we conclude that $N$ has finite moments. \qed
\end{proof}

\noindent We end this section with the final result below. That requires essential infimum of at least one entry of each row of $\lambda$ to be nonzero.
\begin{theo}\label{merg_suff}
Assume (A1)-(A4) and that $X^1$ and $X^2$ are as in Notation \ref{N4.1}. Further assume that for each $k\in \mathcal{X}$ there is at least one $k'\in \mathcal{X}\setminus \{k\}$ such that $\|{\lambda_{kk'}}^{-1}\|_{L^\infty} < \infty$.
Then $X^1$ and $X^2$ eventually merge with probability 1.
\end{theo}
\begin{proof}
Since (A1)-(A4) hold, Theorem \ref{SM_EM} ensures eventual meeting with probability 1. Hence $\mathcal{T}_1<\infty$ with probability 1, where $\mathcal{T}_1$ denotes the first meeting time (see Definition \ref{meet_defi} ). If $\mathcal{T}_1$ is NMT (not a merging time), $X^1$ and $X^2$ separate at the next transition and again due to Theorem \ref{SM_EM}, they meet at $\mathcal{T}_2$, say, which is again finite almost surely. By repeating this argument, if $X^1$ and $X^2$ never merge, we obtain an infinite sequence $\{\mathcal{T}_n\}_n$ of meeting times where each of them are finite almost surely. Using $k_n:= X^1_{\mathcal{T}_n}=X^2_{\mathcal{T}_n}$ and $y_n:= \max( Y^1_{\mathcal{T}_n}, Y^2_{\mathcal{T}_n})$, we get $P\left({\{\mathcal{T}_n \textrm{ is NMT} \}} \mid  \mathcal{F}_{\mathcal{T}_n} \right) = P\left({\{\mathcal{T}_n \textrm{ is NMT} \}} \mid  k_n, y_n \right) = 1-\mathcal{P}(k_n,y_n)$, since  $\{\mathcal{T}_n\}_{n\ge 1}$ is a sequence of stopping times and $(Z^1, Z^2)$ is strong Markov.
Therefore,
\begin{align*}
 E\left[\prod_{n=1}^{N}  \mathds{1}_{\{\mathcal{T}_n \textrm{ is NMT} \}}\right]
= & E\left[\prod_{n=1}^{N-1}  \mathds{1}_{\{\mathcal{T}_n \textrm{ is NMT} \}} E\left[\mathds{1}_{\{\mathcal{T}_N \textrm{ is NMT} \}} \mid  \mathcal{F}_{\mathcal{T}_N} \right]\right]\\
& \le \left(1-\inf_{k\in \mathcal{X}, y\ge 0}\mathcal{P}(k,y)\right)E\left[\prod_{n=1}^{N-1}  \mathds{1}_{\{\mathcal{T}_n \textrm{ is NMT} \}}\right].
\end{align*}
Since the event of never merging can be expressed as $\cap_{n\ge 1}\{\mathcal{T}_n \textrm{ is NMT}\}$, an upper bounded of its probability can be obtained by using the above inequality repeatedly, i.e.,
\begin{align} \label{NoMerge}
P\left(\mathcal{T}_n \textrm{ is NMT},\forall n\ge 1\right) = \lim_{N\to \infty}E\left[\prod_{n=1}^{N}  \mathds{1}_{\{\mathcal{T}_n \textrm{ is NMT} \}}\right] \le  \lim_{N\to \infty}\left(1-\inf_{k\in \mathcal{X}, y\ge 0}\mathcal{P}(k,y)\right)^N.
\end{align}
This confirms that the probability of never merging is zero, provided $\inf_{k\in \mathcal{X}, y\ge 0}\mathcal{P}(k,y) >0$.
Since, (A3) holds, from Theorem \ref{theo4.9},
 \begin{align*}
\mathcal{P}(k,y)\ge & \int_{0}^\infty e^{-\underset{k'\neq k}{\sum}\|\lambda_{kk'}\|_{\infty}y'} \sum_{k'\neq k}\| {\lambda_{kk'}}^{-1}\|^{-1}_{\infty}dy'\\
\textrm{or, }\inf_{k\in \mathcal{X}, y\ge 0} \mathcal{P}(k,y)\ge & \left(\int_{0}^\infty e^{-Cy'} dy'\right)\min_{k\in  \mathcal{X}} \sum_{k'\neq k}\| {\lambda_{kk'}}^{-1}\|^{-1}_{\infty}.
\end{align*}
Since for each $k\in \mathcal{X}$, there is a $k'\in \mathcal{X}\setminus \{k\}$ such that $\|{\lambda_{kk'}}^{-1}\|_{L^\infty} < \infty$, and $\mathcal{X}$ is finite, the right side of above inequality is positive. Thus $\inf_{k\in \mathcal{X}, y\ge 0}\mathcal{P}(k,y) >0$ as desired. \qed
\end{proof}

\section{Conclusion}
\noindent In this paper we make use of a particular type of semimartingale representation of a class of semi-Markov processes. We have then studied various aspects of a pair of solutions having two different initial conditions. Several questions regarding the meeting and merging of stochastic flow of SMP have been answered by considering a solution pair. We have obtained explicit expressions of probabilities of many relevant events in terms of the transition rate matrix.

\noindent The study of eventual meeting and merging in Section 4 is carried out for finite state-space case. These results could be examined for certain infinite state cases, like birth-death processes, or more generally, where all entries of $\lambda$, except $k$ nearest neighbours of diagonal are zero. Apart from this, we also propose another extension. The present study which has been carried out for the time-homogeneous case, can further be  investigated for the time non-homogeneous case. It is clear that the results of Section 4 cannot be extended in a straight forward manner for this general case. We wish to pursue further research in these directions.

\section*{Acknowledgement}
\noindent \noindent The authors are grateful to Professors Mrinal K. Ghosh, and Gopal K. Basak for some helpful discussions. The authors also greatly appreciate some comments from Dr. Subhamay Saha, which helped in improving the readability. The suggestions of an anonymous referee have helped in improving the organization, presentation and focus of this paper significantly.


\begin{thebibliography}{10}

\bibitem{arratia1979coalescing}
{\sc R.~ARRATIA}, {\em
  \href{https://mathscinet.ams.org/mathscinet-getitem?mr=2630231}{Coalescing
  Brownian motions on the line(random walks)[Ph. D. Thesis]}},  (1979).

\bibitem{basak1999sankhya}
{\sc G.~K. Basak, A.~Bisi, and M.~K. Ghosh}, {\em
  \href{https://www.jstor.org/stable/25051226?seq=1}{Stability and Functional
  Limit Theorems for Random Degenerate Diffusions}}, Sankhyā: The Indian
  Journal of Statistics, Series A (1961-2002), 61 (1999), pp.~12--35.

\bibitem{cinlar2011probability}
{\sc E.~{\c{C}}{\i}nlar}, {\em
  \href{https://link.springer.com/book/10.1007/978-0-387-87859-1}{Probability
  and Stochastics}}, Graduate Texts in Mathematics, Springer New York, 2011.

\bibitem{elliott2020semimartingale}
{\sc R.~J. Elliott}, {\em
  \href{https://digitalcommons.lsu.edu/josa/vol1/iss1/1/}{The semi-martingale
  dynamics and generator of a continuous time semi-Markov chain}}, Journal of
  Stochastic Analysis, 1 (2020), p.~1.

\bibitem{Ghosh2009RiskMO}
{\sc M.~K. Ghosh and A.~Goswami}, {\em
  \href{https://epubs.siam.org/doi/10.1137/080716839}{Risk minimizing option
  pricing in a semi-Markov modulated market}}, SIAM J. Control. Optim., 48
  (2009), pp.~1519--1541.

\bibitem{GhoshSaha2011}
{\sc M.~K. Ghosh and S.~Saha}, {\em
  \href{https://www.tandfonline.com/doi/abs/10.1080/07362994.2011.564455?journalCode=lsaa20}{Stochastic
  Processes with Age-Dependent Transition Rates}}, Stochastic Analysis and
  Applications, 29 (2011), pp.~511--522.

\bibitem{generalsemi}
{\sc A.~Goswami, S.~Saha, and R.~K. Yadav}, {\em Semimartingle representation
  of a class of semi-markov dynamics}, 2022.

\bibitem{HARRIS1984187}
{\sc T.~E. Harris}, {\em
  \href{https://www.sciencedirect.com/science/article/pii/0304414984900012}{Coalescing
  and noncoalescing stochastic flows in R1}}, Stochastic Processes and their
  Applications, 17 (1984), pp.~187--210.

\bibitem{levyp}
{\sc P.~Levy}, {\em Systemes semi-markoviens 'aau plus une infinite denombrable
  d'etats possibles}, Proc. Int. Congr.Math., Amsterdam, 2 (1954), p.~294.

\bibitem{AIHP968}
{\sc I.~Melbourne and D.~Terhesiu}, {\em
  \href{https://arxiv.org/pdf/1701.08440v3.pdf}{Renewal theorems and mixing for
  non Markov flows with infinite measure}}, Annales de l'Institut Henri
  Poincaré, Probabilités et Statistiques, 56 (2020), pp.~449 -- 476.

\bibitem{smith1955regenerative}
{\sc W.~L. Smith}, {\em
  \href{https://www.jstor.org/stable/99680?seq=1}{Regenerative stochastic
  processes}}, Proceedings of the Royal Society of London. Series A.
  Mathematical and Physical Sciences, 232 (1955), pp.~6--31.

\bibitem{takacs1954some}
{\sc L.~Tak{\'a}cs}, {\em \href{}{}{Some investigations concerning recurrent
  stochastic processes of a certain type}}, Maygyar Tud. Akad. Mad. Kutat{\'o}
  int. Kozl, 3 (1954), pp.~115--128.

\end{thebibliography}

\end{document}